\documentclass[10pt]{amsart}
\usepackage{enumerate}
\usepackage{amssymb,newlfont,amsmath, amsthm}
\usepackage{caption}
\usepackage{mathtools}
\usepackage{color,graphicx}
\usepackage[update,prepend]{epstopdf}
\usepackage[colorlinks]{hyperref}
\usepackage{enumitem}

\newcommand{\beq}{\begin{equation}}
\newcommand{\eeq}{\end{equation}}
\def\bals#1\eals{\begin{align*} #1 \end{align*}}
\def\bal#1\eal{\begin{align} #1 \end{align}}

\newcommand{\where}{\quad \text{ where } }

\newcommand\Dom\Omega

\newcommand\RR{\mathbb{R}}

\newcommand\Lap\Delta

\newcommand\abs[1]{\left\lvert #1 \right\rvert}

\def\bpde#1\epde{\[\left\{\begin{aligned}#1\end{aligned}\right. \]}
\def\inbpde#1\inepde{\left\{\begin{aligned}#1\end{aligned}\right.}
\def\binpde#1\einpde{\left\{\begin{aligned}#1\end{aligned}\right.}

\newcommand\prn[1]{\left( {#1} \right)}

\def\bbmat{\begin{bmatrix}[r]}
\def\ebmat{\end{bmatrix}}

\newcommand{\barr}{\begin{array}}
\newcommand{\ea}{\end{array}}
\newcommand{\bea}{\begin{eqnarray}}
\newcommand{\eea}{\end{eqnarray}}
\newcommand{\bt}{\begin{table}}
\newcommand{\et}{\end{table}}

\theoremstyle{plain}
\theoremstyle{definition}

\newtheorem{theorem}{Theorem}[section]
\newtheorem{lemma}[theorem]{Lemma}

\theoremstyle{definition}

\theoremstyle{remark}

\numberwithin{equation}{section}

\begin{document}
  
\title{An Elementary Proof that Symplectic Matrices Have Determinant One}

\author{Donsub Rim}
\address{Department of Applied Mathematics, University of Washington, Seattle WA 98195}
\email{drim@uw.edu}

\begin{abstract}
    We give one more proof of the fact that
    symplectic matrices over real and complex fields have determinant one.
    While this has already been proved many times,
    there has been lasting interest in finding an elementary proof 
    \cite{froilan, mackey1}.
    The result is restricted to the real and complex case due to 
    its reliance on field-dependent spectral theory, however in this setting 
    we obtain a proof which is more elementary in the sense that it is direct 
    and requires only well-known facts.
    Finally, an explicit formula for the determinant of conjugate
    symplectic matrices in terms of its square subblocks is given.
\end{abstract}

\maketitle

\section{Introduction}

A \emph{symplectic matrix} is a matrix $A \in \mathbb{K}^{2N \times 2N}$ over 
a field $\mathbb{K}$
that is defined by the property
\beq
    A^T J A = J \quad \text{ where } \quad
J := \begin{bmatrix} O & I_N \\ -I_N & O \end{bmatrix},
    \quad I_N := (\text{identity in } \mathbb{K}^{N \times N}.)
\label{eq:def}
\eeq
We are concerned with the problem of showing that $\det(A) = 1$.
In this paper we focus on the special case when $\mathbb{K} = \mathbb{R}$ 
or $\mathbb{C}$.

It is straightforward 
to show that such a matrix $A$ has $\det(A) = \pm 1$. 
However, there is an apparent lack of an entirely elementary proof 
verifying that $\det(A)$ is indeed equal to $+1$
\cite{froilan, mackey1}, although various proofs have appeared
in the past. Existing proofs
require polar decomposition \cite{meyer} or 
structured polar decomposition \cite{mackey2} in the least.
Symplectic $\mathbb{G}$-reflectors developed in \cite{mackey3}
can be exploited \cite{mackey1} to provide a proof.
A proof utilizing the relationships of linearly independent rows
and columns in the block structure of symplectic matrices, called \emph{the 
complementary bases theorem}, appears in \cite{froilan}.
It is also widely known that a proof results from a study of 
Pfaffians \cite{artin}.
We emphasize that some of these previous proofs, including the 
proof using the complementary bases theorem and that using Pfaffians,
yield the result over arbitrary fields and are therefore more general.
 
The main contribution of this paper is to show that, when restricted to 
real matrices, there is a direct proof that requires only elementary 
linear algebra. With more work, this proof extends to complex matrices, and the
extension only requires basic knowledge of the topology of general linear groups.
Moreover, for the conjugate symplectic matrices defined to satisfy $A^*JA = J$ 
rather than \eqref{eq:def}, the same approach yields an explicit formula 
\eqref{eq:conjformula} for 
the determinant in terms of its four square subblocks.

Although this is an independent result by the author, 
this precise approach for the real case 
had appeared previously in a reference written in Chinese 
by Xu \cite{xu}, to our surprise. 
Perhaps not very widely known, this is the first appearance
of the proof to the best of our knowledge.

\section{A proof that the determinant of a real or complex symplectic matrix 
is one}

Let us denote the set of $2N \times 2N$ 
symplectic matrices over field $\mathbb{K}$ by $Sp(2N,\mathbb{K})$. 

A feature of the proofs below is that they follow a similar overall strategy
for both cases, when $\mathbb{K} = \mathbb{R}$ and $\mathbb{K} = \mathbb{C}$.
Nonetheless, in detail 
there is nontrivial difference arising from field-dependence,
in the form of Lemma \ref{thelem}.

We first give the proof when $A$ is real.

\begin{theorem}
    Let $A \in Sp(2N,\RR)$. Then $\det(A) = 1$.
\end{theorem}

\begin{proof}
Taking the determinant on both sides of $A^T J A = J$,
\begin{equation}
\det(A^T J A) = \det(A^T) \det(J) \det(A) = \det(J).\label{eq:main}
\end{equation}
So we immediately have that $\det(A) = \pm 1$.

Then let us consider the matrix 
\begin{equation}
A^TA + I.  \label{eq:matrix}
\end{equation}
Since $A^TA$ is symmetric positive definite,  
the eigenvalues of \eqref{eq:matrix} are real and greater than $1$.
Therefore its determinant, being the product of its eigenvalues,  
has $\det(A^TA +I) > 1$. 

Now as $\det(A) \ne 0$, $A$ is invertible. Using \eqref{eq:def} we may write
\begin{equation}
    A^TA + I = A^T( A + A^{-T}) = A^T(A + JAJ^{-1}).
\label{eq:factor}
\end{equation}
Denote the four $N \times N$ subblocks of $A$ as follows,
\begin{equation}
A = \begin{bmatrix} A_{11} & A_{12} \\ A_{21} & A_{22} \end{bmatrix},
    \where A_{11},A_{12},A_{21},A_{22} \in \RR^{N \times N}. \label{eq:blocks}
\end{equation}
Then we compute
\begin{align}
    A + JAJ^{-1} &= \begin{bmatrix} A_{11} & A_{12} \\ A_{21} & A_{22} \end{bmatrix}
+ \begin{bmatrix} O & I_N \\ -I_N & O \end{bmatrix}
   \begin{bmatrix} A_{11} & A_{12} \\ A_{21} & A_{22} \end{bmatrix}
\begin{bmatrix} O & - I_N \\  I_N & O \end{bmatrix} \\
  &= \begin{bmatrix} A_{11} & A_{12} \\ A_{21} & A_{22} \end{bmatrix}
+ \begin{bmatrix} A_{22} & -A_{21} \\ -A_{12} & A_{11} \end{bmatrix}
    = \begin{bmatrix} A_{11}+ A_{22} & A_{12} - A_{21} \\ - A_{12}+ A_{21}  & A_{11} + A_{22} \end{bmatrix}.
\end{align}
Writing the blocks as $C := A_{11} + A_{22}$ and $D:= A_{12} - A_{21}$, 
we make use of a unitary transform 
\begin{align}
A + JAJ^{-1} &= \begin{bmatrix} C & D \\ -D & C \end{bmatrix} \notag \\
             &=  
    \frac{1}{\sqrt{2}}\begin{bmatrix} I_N & I_N \\ iI_N & -iI_N \end{bmatrix}
             \begin{bmatrix} C + i D & O \\ O & C - i D \end{bmatrix}
    \frac{1}{\sqrt{2}} \begin{bmatrix} I_N & -iI_N \\ I_N & iI_N \end{bmatrix}.
\label{eq:unitary}
\end{align}
We plug this factorization into \eqref{eq:factor}.
Note that $C,D$ are both real. This allows the complex 
conjugation to commute with the determinant (as it is a polynomial of its 
entries)
\bals
 0 < 1 < \det(A^TA + I)  &= \det(A^T(A + JAJ^{-1})) \\
&= \det(A) \det(C + i D) \det(C - iD) \\
&= \det(A) \det(C + i D) \det\prn{\overline{C + iD}}\\
&= \det(A) \det(C + iD) \overline{\det(C + iD)} 
= \det(A) \abs{\det(C + iD)}^2.
\eals
Clearly, none of the two determinants on the RHS can be zero, 
so we may conclude
$\abs{\det(C + iD)}^2 > 0$. Dividing this through on both sides,
we have $\det(A) > 0$, and thus $\det(A) = 1$.
\end{proof}
The unitary transform \eqref{eq:unitary} appears in the final part
of a proof that the determinant of any real symplectic matrix is $+1$
using polar decomposition in \cite{meyer} although
there it is not used in conjunction with \eqref{eq:matrix}.
In our proof its use was in showing that
\begin{equation}
\det\left(\begin{bmatrix} C & D \\ -D & C \end{bmatrix}\right) \ge 0
\quad \text{ for any }
C, D \in \RR^{N \times N}. \label{ineq:real}
\end{equation}
The next lemma is a generalization of \eqref{ineq:real} to complex matrices,
and is the main hurdle in applying the above approach
to the complex field.
It arises naturally in the treatment of quaternionic matrices,
and was proved in \cite{aslaksen, zhang}. We simply reproduce
the latter proof here for completeness.
\begin{lemma} \label{thelem}
Let $C, D \in \mathbb{C}^{N \times N}$. Then
$\det\left(\begin{bmatrix} C & D \\ - \overline{D} & \overline{C} \end{bmatrix} \right) \ge 0. $
\end{lemma}

\begin{proof}
    Let us first assume that $C$ is invertible. 

    Then letting $E:=C^{-1}D$, $\det\prn{\overline{E}E + I} \ge 0$
    since eigenvalues of $\overline{E}E$ occur in conjugate
    pairs and every negative eigenvalue of $\overline{E}E$ has even
    algebraic multiplicity \cite{HJ}.
    Note that $E$ and $I$ commute,
    which implies that
    \[
    \det \prn{ \begin{bmatrix} I & E \\ -\overline{E} & I \end{bmatrix} } \ge 0.
    \]
    Then we have,
    \bals
\begin{bmatrix} C^{-1} & 0 \\ 0 & \prn{\overline{C}}^{-1} \end{bmatrix}
\begin{bmatrix} C & D \\ - \overline{D} & \overline{C} \end{bmatrix}
    =
\begin{bmatrix} I & C^{-1} D \\ - \overline{C^{-1}D} & I \end{bmatrix}
    =
\begin{bmatrix} I & E \\ - \overline{E} & I \end{bmatrix}.
    \eals
    Taking the determinant on both sides, we have the result.
    The Lemma follows by continuity.
\end{proof}

Now we prove the complex case. Given Lemma \ref{thelem} the proof is
very similar to the real case.

\begin{theorem}
Let $A \in Sp(2N,\mathbb{C})$. Then $\det(A) = 1$.
\end{theorem}

\begin{proof}
If $A$ is symplectic then
we again have from \eqref{eq:main} that $\det(A) = \pm 1$.
Consider the matrix
\begin{equation}
A^* A + I \label{eq:cspd}.
\end{equation}
The matrix $A^*A$ is Hermitian positive definite and therefore
the matrix \eqref{eq:cspd} has eigenvalues greater than one,
and thus $\det(A^*A + I) > 1$. Then we factor,
\begin{equation}
A^*A + I = A^*(A + A^{-*}) = A^* (A + \overline{A^{-T}}).
\label{eq:cdecom}
\end{equation}
Then we again use the definition \eqref{eq:def} and
with the notation \eqref{eq:blocks} for subblocks,
\bals
A + \overline{JAJ^{-1}} &= 
\begin{bmatrix}
A_{11} & A_{12} \\ A_{21} & A_{22} 
\end{bmatrix}
+ 
\overline{ \begin{bmatrix} A_{22} & -A_{21} \\ -A_{12} & A_{11} \end{bmatrix}}  \\
&= \left[
\begin{array}{rr} A_{11} + \overline{A_{22}} &  A_{12}- \overline{A_{21}} \\ - \overline{A_{12}} +A_{21}& \overline{A_{11}} + A_{22}
\end{array}\right]
= \begin{bmatrix} C & D \\ -\overline{D} & \overline{C} \end{bmatrix},
\eals
if we let $C := A_{11} + \overline{A_{22}}$ and $D:= A_{12} - \overline{A_{21}}$.
Therefore we have from \eqref{eq:cdecom} 
\bals
0 < 1 &< \det(A^*(A + A^{-*})) = 
\overline{\det(A)} 
\det\left( 
\begin{bmatrix} C & D \\ - \overline{D} & \overline{C} \end{bmatrix} \right)
=
\det(A) 
\det\left( 
\begin{bmatrix} C & D \\ - \overline{D} & \overline{C} \end{bmatrix} \right).
\eals
By Lemma \ref{thelem} we necessarily have that  
$\det\left( 
\begin{bmatrix} C & D \\ - \overline{D} & \overline{C} \end{bmatrix} \right)>0$,
so we divide and conclude $\det(A) > 0$, showing that $\det(A) = 1$.
\end{proof}

\section{An explicit formula for determinant of conjugate symplectic matrices}

Conjugate symplectic matrices $A \in \mathbb{C}^{2N \times 2N}$ 
are defined similarly to symplectic matrices, by replacing the transpose
in \eqref{eq:def} with the conjugate transpose,
\[
A^*JA = J.
\]
Again, it immediately follows that $\abs{\det(A)} = 1$,
however $\det(A)$ may take on any complex value on the unit circle
\cite{mackey1}.
Turning our attention to \eqref{eq:cspd} once more,
\[
A^*A +I = A^*(A + A^{-*}) = A^*(A + J AJ^{-1}).
\]

With the previous notation for the $N \times N$ subblocks of $A$ and 
with $C,D$ defined as 
\[
A = \begin{bmatrix} A_{11} & A_{12} \\ A_{21} & A_{22} \end{bmatrix},
    \quad
    C := A_{11} + A_{22},
    \quad
    D := A_{12} - A_{21},
\]
we obtain \eqref{eq:unitary} but for a complex matrix.
Then it follows that
\begin{align*}
1 < \overline{\det(A)} \det\left( 
\left[\begin{array}{rr} C & D \\ -D & C \end{array}\right] \right)
&= \overline{\det(A)} \det(C + iD) \det(C - iD) \\
&= \det(A)\overline{\det(C^2 + D^2 - i[C,D])} ,
\end{align*}
where $[C,D] := CD - DC$ is the commutator.

This determines the phase and therefore $\det(A)$, yielding the formula
\begin{equation}
    \det(A) = \frac{\det((A_{11} + A_{22})^2 + (A_{12} - A_{21})^2 - i[A_{11} + A_{22},A_{12} - A_{21}])}{\abs{\det((A_{11} + A_{22})^2 + (A_{12} - A_{21})^2 - i[A_{11} + A_{22},A_{12} - A_{21}])}}.
\label{eq:conjformula}
\end{equation}

\section{Acknowledgements}
The author would like to express thanks to 
Prof. Froil\'an M. Dopico and Prof. Ioana Dumitriu for carefully reading
this paper and providing valuable comments, 
and thanks Prof. Bernard Deconinck for 
introducing him to this problem. 

He also thanks Dr. Meiyue Shao for bringing to attention the
reference \cite{xu}.

\bibliographystyle{amsplain}

\end{document}